\documentclass[10pt, oneside]{article}
\usepackage{amsthm}
\usepackage{amsmath}
\usepackage{amsfonts}
\input xypic

\newtheorem{thm}{Theorem}
\newtheorem{cor}[thm]{\bf{Corollary}}
\newtheorem{lem}[thm]{Lemma}

\newcommand{\ds}{\displaystyle}

\begin{document}

\begin{center}
    \textbf{ THE DIMENSIONS OF LU(3,q) CODES} \footnote[1]{This work was supported by Chat Yin Ho scholarship of Department of Mathematics at University of Florida.}\\
     \textsc{Ogul Arslan} \\
      \textsc{Department Of Mathematics}\\
       \textsc{University Of Florida}
\end{center}

\small
\textsc{ABSTRACT:}
A family of LDPC codes, called $LU(3,q)$ codes, has been constructed
from q-regular bipartite graphs.  Recently, P. Sin and Q. Xiang
determined the dimensions of these codes in the case that q is a power
of an odd prime.  They also obtained a lower bound for the dimension of
an $LU(3,q)$ code when q is a power of 2.  In this paper we prove
that this lower bound is the exact dimension of the $LU(3,q)$ code. The
proof involves the geometry of symplectic generalized quadrangles, the
representation theory of $Sp(4,q)$, and the ring of polynomials.\\

\normalsize
\begin{center}
    \textsc{1. Introduction}
\end{center}

Let $ P^*$ and $L^*$ be two sets in bijection with
 $\mathbb{F}_q^3$, where $q$ is any prime power. In [4], an element
 $(a,b,c) \in P^*$ is defined to be incident with an element
$[x,y,z] \in L^*$ if and only if $y=ax+b$ and $z=ay+c$. The binary
incidence matrix with rows indexed by $P^*$ and columns indexed by
$L^*$ is denoted by $H(3,q)$. The two binary codes having
$H(3,q)$ and its transpose as parity check matrices are called
$LU(3,q)$ codes in [4].

Let $V$ be a 4 dimensional vector space over the field
$\mathbb{F}_q$ of $q$ elements. We assume that $V$ has a nonsingular
alternating bilinear form $(v, v^{'})$, that is,  $(v, v^{'})$ is
linear in both components and $(v,v)=0$ for all $v$. Let
$Sp(4,q)$ be the symplectic group of linear automorphisms
preserving this form. We pick a symplectic basis $e_0,\, e_1,\, e_2,\,
e_3$ of $V$, with $(e_i,e_{3-i})=1$ for $i=0,1.$

We denote by $P$, the projective space $P(V)$, the space of one
dimensional subspaces of $V$. These one dimensional subspaces are
called the points of $P$. A subspace of V is called \textit{totally isotropic},
if $(v,v^{'})=0$ whenever $v$ and $v^{'}$ are both in the
subspace. We let $L$ be the set of totally isotropic 2-dimensional
subspaces of $V$, considered as lines in $P$. The pair $(P,L)$, with
the natural relation of incidence between the points and lines is
the symplectic generalized quadrangle $W(q)$. In this paper the term \textit{line} will
always mean an element of $L$.  One can see that
given any line $\ell$ and a point $p$ not on that line there is a
unique line that passes through $p$ and intersects $\ell$.

Fix a point $p_0 = \langle e_0\rangle \in P $ and a line $\ell_0 =
\langle e_0 , e_1 \rangle \in L$. For a point $p \in P$, we define
$p^{\perp}$ to be the set of points on all the lines that pass
through $p$. Thus, $p_0^{\perp}=\{(a:b:c:0)| \, a, b,  c \in
\mathbb{F}_q \}$ where $(a:b:c:d)$ are the homogeneous coordinates
of a point. Let $P_1$ be the set of points not in $p_0^{\perp}$ and $L_1$ be the set
of lines which do not intersect $\ell_0$. Hence other incidence systems of
interest are $(P_1,L_1)$, $(P,L_1)$ and $(P_1,L)$. Let
 $M(P,L)$ be the incidence matrix whose rows are indexed by $P$,
and the columns by $L$. Similarly, we get the incidence matrix
$M(P_1,L_1)$, which can be thought as a  submatrix of $M(P,L)$. It was proven in [8, appendix] that the incidence systems $(P^*,L^*)$ and $(P_1,L_1)$ are equivalent. Hence, $M(P_1,L_1)$ and its transpose are parity check matrices for $LU(3,q)$ codes.

The 2-ranks of $M(P,L)$ and $M(P_1,L_1)$ for $q$ a power of an odd
prime, were proven to be $\ds (q^3+2q^2+q+2)/2$ and
$\ds (q^3+2q^2-3q+2)/2$ in [1, theorem 9.4] and [8, theorem 1.1]
respectively.

The formulas for the case where $q$ is a power of 2 are quite
different. It was proven in  [7, theorem 1] that the 2-rank of
$M(P,L)$ is
$\ds 1+\left[ (1+ \sqrt {17} \,)/2\right]^{2t}+ \left[ (1- \sqrt
{17}\,)/2\right]^{2t}.$

In this paper we prove the following theorem. The formula in the theorem was conjectured in [8] based on the computer calculations of
J.-L. Kim.

\begin{thm} Assume $q=2^t$ for some positive integer $t$. The
2-rank of $M(P_1,L_1)$ is\\
$$ 1+\left( \frac{1+ \sqrt {17}}{2}\right)^{2t}+ \left( \frac{1- \sqrt {17}}{2}\right)^{2t} -2^{t+1}.$$
\end{thm}

Hence we get the following corollary.

\begin{cor} The dimension of the $LU(3,q)$ code for $q$ a power of 2 is \\ $$
2^{3t}+2^{t+1}-1-\left(\frac{1+\sqrt{17}}{2}\right)^{2t}-\left(\frac{1-\sqrt{17}}{2}\right)^{2t}.$$
\end{cor}

The dimension of the $LU(3,q)$ code for $q$ a power of an odd prime  was
proven to be $ (q^3-2q^2+3q-2)/2$ in [8, Corollary 1.2].

For the rest of the section we can assume that $q$ is an arbitrary
prime power.

We denote by $\mathbb{F}_2[P]$ the space of $\mathbb{F}_2$ valued
functions on $P$. We can think of elements of $\mathbb{F}_2[P]$ as
$q^3+q^2+q+1$ component vectors whose entries are indexed by the
points of $P$ so that for any function $f$, the value of each
entry is the value of $f$ at the corresponding point. The
characteristic function $\chi_p$ for a point $p \in P$ is the
function whose value is 1 at $p$, and zero at any other point. Thus,
$\chi_p$ is the $q^3+q^2+q+1$ component vector whose entry that
corresponds to $p$ is 1, and all the other entries are zero. The
characteristic functions for all the points in $P$ form a basis for
$\mathbb{F}_2[P]$. For any line $\ell \in L$, the characteristic
function $\chi_{\ell}$ is the function given by the sum of the $q+1$
characteristic functions of the points of $\ell$. The subspace of
$\mathbb{F}_2[P]$ spanned by all the $\chi_{\ell}$ is the
$\mathbb{F}_2$ code of $(P,L)$, denoted by $C(P,L)$. We can think
of it as the column space of $M(P,L)$. Most of the time we will
not make a distinction between the lines and the characteristic
functions of the lines. For example, we will say, let $C(P,L_1)$
be the subspace of $\mathbb{F}_2[P]$ spanned by the lines of
$L_1$. Let $C(P_1, L_1)$ denote the code of $(P_1,L_1)$ viewed as
a subspace of $\mathbb{F}_2[P_1],$ and let $C(P_1,L)$ be the
larger subspace of $\mathbb{F}_2[P_1]$ spanned by the restrictions
to $P_1$ of the characteristic functions of all lines of $L$.

We consider the natural projection map
$\pi_{P_1} : \mathbb{F}_2[P] \rightarrow \mathbb{F}_2[P_1]$
given by the restriction of functions to $P_1$. We denote its
kernel by $ker\,\pi_{P_1}$.

Let $Z \subset C(P,L_1)$ be a set of characteristic functions of
lines in $L_1$ which maps bijectively under $\pi_{P_1}$ to a basis
of $C(P_1, L_1)$. Let $X$ be the set of characteristic functions
of the $q+1$ lines passing through $p_0$, and let $X_0= X
\setminus {\ell_0}.$ Furthermore, we pick $q$ lines that intersect
$\ell_0$ at $q$ distinct points except $p_0$, and call the set
of these lines as $Y$. These sets $X, Y$, and $Z$ are disjoint,
also note that $X \subset ker\, \pi_{P_1}.$

The following lemma and corollary were proven in [8].
\begin{lem} $ X_0 \cup Y \cup Z$ is linearly independent over
$F_2$.
\end{lem}

Hence, $|X_0 \cup Y | =2q $, while
$|Z|=dim_{\,\mathbb{F}_2}\,C(P_1,L_1)$.

\begin{cor} Let $q$ be an arbitrary prime power. Then
$dim_{\,\mathbb{F}_2}LU(3,q) \geq q^3 -dim_{\,\mathbb{F}_2}C(P,L)
+2q.$

\end{cor}

The proof of Theorem 1 follows from Lemma 3 and the dimension of $C(P,L).$  In section 2 we prove that $
 X_0 \cup Y \cup L_1$ spans $C(P,L)$. Then we show in section 3 that the span of $
 X_0 \cup Y \cup L_1$ and $
 X_0 \cup Y \cup Z$ are the same.

\begin{center}
    \textsc{2. The Grid Of Lines}\\
\end{center}
Unless otherwise is stated we assume that $q=2^t$ for the rest of the paper.
\begin{lem} Let $ \ell $ and $ \ell ^{'}$ be two lines passing through
$p \in \ell _{0} . $ Then $ \chi_{\ell} +\chi_{\ell^{'}} \in C(P,
L_1).$
\end{lem}

\begin{proof}
The points of the quadrangle $W(q)$ are regular as it is defined in [6, section 1.3, p.4]. When $q$ is even this quadrangle is known to be self-dual [6, 3.2.1]. Hence, the lines of $W(q)$ are regular for the case of even $q$. Thus one can show that there is a grid of lines between $\ell$ and $\ell^{'}$. This means there are two sets of lines $\Delta$ and $\Lambda$ such that each set has $q$ elements, each line in $\Delta$ intersects $\ell \setminus \{p\}$ and distinct lines of $\Delta$ intersects $\ell \setminus \{p\}$ in distinct points. Similarly, each line in $\Lambda$ intersects $\ell^{'} \setminus \{ p\}$ and distinct lines of $\Lambda$ intersects $\ell^{'} \setminus \{p\}$ in distinct points. Moreover, every line of $\Delta$ intersects every line of $\Lambda$.\\

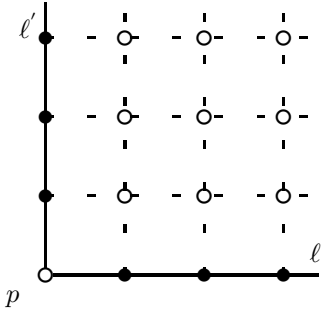
\begin{picture}(100,140)

\thicklines

 \put(0,110){$\ell^{'}$}
 \put(110,25){$\ell$}
 \put(-5,10){$p$}

 \put(13,20){\line(1,0){100}} 
 \put(10,23){\line(0,1){100}} 

 \multiput(40,20)(30,0){3}{\circle*{5}}
 \multiput(10,50)(0,30){3}{\circle*{5}}
 \multiput (40,50) (0,30){3}{\circle{5}}
 \multiput (70,50) (0,30){3}{\circle{5}}
 \multiput (100,50) (0,30){3}{\circle{5}}
 \put(10,20){\circle{5}}

 \multiput(40,20)(0,16){7}{\line(0,1){3}} 
 \multiput(70,20)(0,16){7}{\line(0,1){3}} 
 \multiput(100,20)(0,16){7}{\line(0,1){3}} 

 \multiput(10,50)(16,0){7}{\line(1,0){3}} 
 \multiput(10,80)(16,0){7}{\line(1,0){3}} 
 \multiput(10,110)(16,0){7}{\line(1,0){3}} 
\end{picture}

We add characteristic functions of these lines and get

$$ \sum_{\gamma \in \Delta \cup \Lambda} \chi_{\gamma}=
\chi_{\ell} + \chi_{\ell^{'}} \in C(P, L_1).$$

\end{proof}

\begin{lem} For any choice of $Y$, $\ell \in L \setminus \{{\ell_0}\} $
and \textbf{1} are in the span of $ X_0 \cup Y \cup L_1 $.
\end{lem}

\begin{proof} It is enough to show that any line $\ell$ in $L \setminus (X \cup L_1)$ is in the span of $ X_0 \cup Y \cup L_1 $. It is immediate that $\ell$
intersects $\ell_0$ at a point $p$ other than $p_0$. Let $\ell^{'}$ be the line in
$Y$ that intersect $\ell_0$ at $p$.
  Then, by the previous result $\chi_{\ell}+\chi_{\ell^{'}}
$ is in the span of $ L_1$. Thus $
( \chi_{\ell}+\chi_{\ell^{'}} ) +\chi_{\ell^{'}}= \chi_{\ell}$ is in
the span of $Y \cup L_1$. Thus any line in $ L \setminus
\{\ell_0\}$ can be written as a linear combination of the lines in
$X_0 \cup Y \cup L_1.$

In order to prove the second part of the lemma, we pick a line in $L_1$, say $\ell^{*}$.
Since $\ell^*$ does not intersect $\ell_0$,  all the lines that
intersect $\ell^*$ are in $\langle X_0, Y, L_1 \rangle $. Hence we add all these lines, including $\ell^*$, to get \textbf{1}.

\end{proof}

\begin{lem}  $\ell_0$ is contained in the span of $X_0 \cup Y \cup L_1.$
\end{lem}

\begin{proof} $$ \chi_{\ell_0} = \textbf{1} + \sum_{\ell \cap \ell_0
\not = \emptyset , \ell \not = \ell_0} \chi_{\ell} \in \langle
X_0, Y, L_1 \rangle.$$
\end{proof}

Thus any line $\ell \in L$ is in the span of $X_0 \cup Y \cup
L_1$. It remains to show the span of $X_0 \cup Y \cup L_1$ is the
same as the span of $X_0 \cup Y \cup Z.$\\
In the next section we introduce a new way of representing the
lines of $P$.\\

\begin{center}
\textsc{3. The Polynomial Approach}\\
\end{center}
Let $k$ denote the field $ \mathbb{F}_q $. Consider the space, $k[V]$, of $k$-valued functions on $V$, where
the elements of this space are  vectors with  $q^4$ components on $k$.

Let $ R= k[x_0,x_1,x_2,x_3]$, be the ring of polynomials in four
indeterminates. We can think of any polynomial in R as
a function in $k[V].$ In order to find the value of $f(x_0,x_1,x_2,x_3) \in R$ at
$v=(a_0,a_1,a_2,a_3) \in V$ we just substitute $x_i$ with $a_i$ for all
$i$. Thus, there is an homomorphism from $R$ to $k[V]$ that maps
every polynomial to a function. One can prove that this homomorphism is in fact an isomorphism
between $R/I$ and $k[V]$, where $I$ is the ideal generated by $\{
(x_0^q-x_0), (x_1^q-x_1), (x_2^q-x_2), (x_3^q-x_3) \} $.

For each $f + I \in R/I$, there is a unique polynomial
representative $f^* \in R$ such that each indeterminate in $f^*$
is of degree less than or equal to $q-1$ and $f + I = f^* +I$. Let
$R^*$ be the set of all such representatives.
By a \textit{term} of an element $f+I $ of $R/I$ we mean a monomial of its
representative $f^*$ in $R^*$.

Let $k[V \setminus \{0\}] $ be the space obtained by restricting
functions of $k[V]$ to $V\setminus \{0\}$, and $k[V \setminus
\{0\}]^{k^{\times}} $ be the subspace of $k[V \setminus \{0\}]$
fixed by $k^{\times}$. In other words, $k[V \setminus
\{0\}]^{k^{\times}} $ is the space of functions $f$ in $k[V \setminus \{0\}]$ such that $f(\lambda v )=f(v)$ for every $ v \in V \setminus \{0\}$, and  $\lambda \in k^{\times}$. Thus, for each $p = \langle v \rangle
\in P $ the value of $f$  on $p \setminus \{0\}$ will be constant.
Hence $f$ can be thought as a function on $P$. On the other hand,
any function $f \in k[P]$ can be extended to a function $\bar{f}
\in k[V \setminus \{0\}]^{k^{\times}} $ by defining the value of
$\bar{ f }(v)$ to be the same as $f(p)$, where $p$ is the point so that $v \in p$. Thus,
there is a one to one correspondence between $k[P]$ and $k[V
\setminus \{0\}]^{k^{\times}}$, and $k[P]$ can be embedded into
$k[V]^{k^{\times}}$.

Since $k[V] \simeq R/I $, there is a space $R_P$ which is
isomorphic to  $k[P]$, and that can be embedded in to $(R /
I)^{k^{\times}}$. Elements of $R_P$ are classes of polynomials.
Let $R_P^* \subseteq R^*$ be the set of representatives of
elements of $R_P$. For any element $g+I$ of $R_P$ the unique
representative $g^*$ in $R_P^*$ will be a homogeneous polynomial
whose terms have degrees which are multiples of $(q-1)$. In this case, the
set of monomials of the form
$x_0^{m_0}x_1^{m_1}x_2^{m_2}x_3^{m_3}$ in $R^*_P$ where $m_0+m_1+m_2+m_3$ is a multiple of $(q-1)$
 will map to a basis of $R_P.$
Since these monomials are in $R_P^*$, each $m_i
\leq q-1.$

For a point $p \in P$, let $\delta _p^* $ be the polynomial in
$R_P^*$ that corresponds to the characteristic function $\chi _p$
of $p$ in $k[P]$. So,

\[ \delta_p^*(v) = \left\{ \begin{array}{ll}
         1 & \mbox{if $\langle v \rangle=p$},\\
        0 & \mbox{if $\langle v \rangle \not = p$}.\end{array} \right.  \]

For a line $ \ell \in L$, let $\delta _{\ell}^* $ be the
polynomial in $R_P^*$ that corresponds to the characteristic
function $\chi _{\ell}$ of $\ell$ in $k[P]$. So,

\[ \delta_{\ell}^*(v) = \left\{ \begin{array}{ll}
         1 & \mbox{if $\langle v \rangle \in \ell$},\\
        0 & \mbox{if $\langle v \rangle \not \in \ell$}.\end{array} \right.  \]

\textbf{Example:} Let $\ell_0 = \langle (1:0:0:0),(0:1:0:0)
\rangle$, then $\delta^*_{\ell_0}=(1+x_2^{q-1})(1+x_3^{q-1})$
would be the characteristic function for $\ell_0$.

The symplectic group $Sp(4,q)$ acts transitively on the characteristic functions of the lines of $L$,
so it also acts transitively on the classes of characteristic
functions of lines in $R_P$.
Hence, by applying the elements of $Sp(4,q)$ to $\delta_{ \ell_0 }^*$,
we can obtain all $q^3+q^2+q+1$
 polynomials corresponding to the characteristic functions of lines of $L$.
The code $C(P,L)$ is spanned by the classes of these polynomials.
So $C(P,L)$ is spanned by the classes of polynomials of the form
$ (1+ (\sum_{i=0}^3
a_ix_i)^{q-1})(1+ (\sum_{i=0}^3 b_ix_i)^{q-1}) + I \textit{,
where } a_i, b_i \in k$  such that the 2-dimensional subspace of $V$ given by
$ a_0x_0+a_1x_1+a_2x_2+a_3x_3=0$ and $b_0x_0+b_1x_1+b_2x_2+b_3x_3=0$ is a line in $L$.
Therefore for $c + I \in C$, $c^*$ is a homogeneous polynomial
whose terms have degrees $0, q-1$ or $ 2(q-1)$. We also note that
the degree of any variable in $c^*$ must be less than or equal to
$q-1$.\\

3.1. \textbf{Another way of representing the polynomials in $R^*$ :}\\

The method of this section was first introduced in [2].

\textbf{Definition:} We call a polynomial $f \in R^*$
\textit{digitizable} if it is possible to find square free
homogeneous polynomials, $f_i$, called \textit{digits} of $f$, so that
$f=f_0f_1^2f_2^{2^2} \ldots f_{t-1}^{2^{t-1}}$. In this case, we
denote $f$ as $ [f_0,f_1, \ldots f_{t-1} ]$, and call this
notation the 2-adic t-tuple of $f$.

\textbf{Example:} Every monomial $m =
x_0^{m_0}x_1^{m_1}x_2^{m_2}x_3^{m_3}$ in $R^*$ is digitizable.
Since each $m_i \leq q-1$, we can find $n_{i,j} \in \{0, 1\}$ such
that;
$$ m_i =   n_{ i, 0} + 2 n_{ i, 1} + 2^2 n_{ i, 2} + \ldots + 2
^{t-1} n_{ i, t-1} \text{ \hspace{5mm} for all } i .$$
The 2-adic t-tuple for $m$ is $[f_0, f_1, \ldots , f_{t-1}]$ where
$ f_i = x_0^{ n _{0,i} }x_1^{ n _{1,i} }x_2^{ n _{2,i} }x_3^{ n
_{3,i} } \text { \hspace { 5mm} for all } i. $

\textbf{Example:} For $q=8$, $f =
x_0^3x_1x_3^6+x_0x_1^3x_2^2x_3^4$ is digitizable with digits $ f_0
= x_0x_1, f_1= x_0x_3+x_1x_2 , f_2= x_3 $. Note that,
\begin{eqnarray*}
 f &= &[x_0x_1,x_0x_3+x_1x_2, x_3] \\
   &= & [x_0x_1, x_0x_3, x_3] + [x_0x_1,x_1x_2,x_3]\\
\end{eqnarray*}

Let $\beta := \{ [f_0,f_1, \ldots , f_{t-1}] +I | \,\, f_i \in \{ 1, x_0,
x_1, x_2, x_3, x_0x_1, x_0x_2, x_1x_3, x_2x_3,x_0x_1x_2, \\ x_0x_1x_3,x_0x_2x_3,x_1x_2x_3, x_0x_3+x_1x_2 \}
\}$

\begin{lem} The code $C(P,L) $ lies in the  span of $\beta$.
\end{lem}

\begin{proof} This just a special case of the theorem 5.2 in [2] with m=2 and r=2.

\end{proof}

 3.2. \textbf{The kernel:}\\

$k[P_1]$ is the space of $k$ valued functions on $P_1$. Let
$R_{P_1}$ be the space of classes of polynomials that corresponds
to $k[P_1]$. As before we use $R^*_{P_1}$ to denote the set of
unique representatives of elements of $R_{P_1}$.

In this section
we will find the dimension of $C(P,L) \cap ker\, \pi_{P_1}$, where
$\pi_{P_1}: R_{P} \rightarrow R_{P_1}$ is the projection map.
Elements of $ker\, \pi_{P_1}$ are the classes of polynomials whose
values at the points of $P_1$ are zero. Any element of the
form $(1+x_3^{q-1})f + I$ is in the kernel. On the other hand,
$f+I= (x_3^{q-1}+1)f+I $ for any class $f+I \in ker\, \pi_{P_1} $. This
is because for any point $p$, the value of $(x_3^{q-1}+1)f$ is zero if
$p \in P_1$, and $f(p)$ otherwise.

\begin{lem} Any element of $ker\, \pi_{P_1} $ can be written in the form
$(1+x_3^{q-1})h +I$ where $h$ is in $R_P^*$ and $h$ does not contain
indeterminate $x_3$.
\end{lem}
\begin{proof} Let $(x_3^{q-1}+1)f+I$, $f \in R_p^*$ be an element of $ker \pi_{P_1} $.
 Since $x_3^q=x_3$, we get $x_3^{q-1}(x_0^ix_1^jx_2^kx_3^l) +I =
x_0^ix_1^jx_2^kx_3^l +I$, for $l \geq 1$ . Thus, any term of $f+I$
that contains $x_3$ is invariant under multiplication by
$x_3^{q-1}$. Hence, the terms with $x_3$ will disappear in the
expansion
$( x_3^{q-1}f+f) +I.$
So, we can find a polynomial $h$ without indeterminate $x_3$ and $(x_3^{q-1}+1)f+I=(x_3^{q-1}+1)h+I$.

\end{proof}

For the rest of the section we fix  an element $r + I$ of $ker \pi_{P_1} \cap
C(P,L) $. Let $r^*$ be its unique representative in $R_P^*$.
Since $r^*+I $ is in the kernel, $r^* = (1 +
x_3^{q-1})h(x_0,x_1,x_2) \text{ for some } h \in R_P^*$ . Since
$r^*+I$ is also in $C(P,L)$,  it is in the span of $\beta$, and
its terms have degrees $0, q-1 $ or $2(q-1)$.

\begin{lem} The degree of the digits of any non-constant monomial of $h$
 is 1.
\end{lem}
\begin{proof}Let $m$ be a non-constant monomial of $h$. Then $m=[g_0, g_1, \ldots , g_{t-1}]$ for some $g_i=
x_0^{n_{0,i}} x_1^{n_{1,i}}x_2^{n_{2,i}}$, where $n_{j,i}\in
\{0,1\}.$ Let $deg(g_i)=k_i$ for each $i$. Hence
$x_3^{q-1}m= [x_3g_0, x_3g_1, \ldots , x_3g_{t-1}]$ is a t-tuple
of a monomial of $r^*$. Since $r^*+I$ is in the span of $\beta$,
the digits of $x_3^{q-1}m$ cannot have degrees greater than 3. Thus, $k_i= 0$, $ 1$, or $2$  for each $i$. \\
Since $r^* +I$ is in $C(P,L)$, and $x_3^{q-1}m$ is a monomial of $r^*$, the degree of $x_3^{q-1}m$ is $q-1$ or $2(q-1)$. Since $m$ is nonconstant,  $
deg(m)= q-1.$ Hence, $k_0+2k_1+ \ldots
 2^{t-1}k_{t-1} = 2^t-1 .$ Since $2^t-1$ is an odd number,
 $k_0=1$. Then we get $ k_1+2k_2+\ldots +
 2^{t-2}k_{t-1}=2^{t-1}-1$ and so $k_1=1$. We repeat this process
 until we get $k_i=1$ for all $i$.

\end{proof}

\begin{lem} $h$ is in the span of the set $\{[1,1, \ldots ,1]\} \cup \{ [g_0, \ldots ,
g_{t-1}]| \text{ } g_i \in \{ x_1, x_2\}, \text{ for } 0 \leq i
\leq t \}.$
\end{lem}
\begin{proof} It is enough to show that $h$ does not contain
 the variable $x_0$.

 Suppose one of the monomials, say $[g_0, \ldots ,
g_{t-1}]$, of $h$ has $x_0$ in it.
 So $g_i=x_0$ for some $i$.
 Then, $x_3^{q-1}[g_0,g_1, \ldots, x_0,\ldots ,
g_{t-1}]=[g_0x_3,g_1x_3, \ldots, x_0x_3,\ldots ,
g_{t-1}x_3]$ is a monomial in $r^*$. We know that $r^*$ is a linear combination
of the elements of $\beta$, so, the coefficient of
$[g_0x_3,g_1x_3, \ldots, x_0x_3 + x_1x_2,\ldots , g_{t-1}x_3]$
is non zero.  Hence, $r^*$ contains the monomial
$[g_0x_3,g_1x_3, \ldots, x_1x_2,\ldots , g_{t-1}x_3]$ also.
Note that the degree of $x_3$ in this monomial is  different from $0$ or  $q-1$.
However this is impossible since $r^* = x_3^{q-1}h+h$, the degree of $x_3$ in any monomial of $r^*$ is either $0$ or $q-1$.

\end{proof}
\begin{cor} $dim(ker \pi_{P_1} \cap C(P,L)) = q+1.$
\end{cor}
\begin{proof} Since $X \subseteq ker \pi_{P_1}\cap C(P,L)$, and
elements of $X$ are linearly independent,  $dim(ker \pi_{R_{P_1}}
\cap C) \geq q+1.$

Any element of $ker \pi_{R_{P_1}}\cap C(P,L)$ is of the form $
(1+x_3^{q-1})h + I$ , where, by the previous lemma, $h$ lies in
space of dimension at most $q+1$. Thus, $dim(ker \pi_{P_1} \cap C(P,L))
= q+1.$

\end{proof}
Following lemma was proven in [8], the proof works the same for
the even case also.\\
\begin{lem} $ker \pi_{P_1} \cap C(P,L_1)$ has dimension $q-1$
and having as basis the set of functions $\chi_{\ell}-\chi_{\ell^{'}}$ where
$\ell \not = \ell_0$ is an arbitrary but fixed line through $p_0$
and $\ell^{'}$ varies over the $q-1$ lines through $p_0$ different
from $\ell_0$ and $\ell$.
\end{lem}

\begin{proof} By lemma 5 applied to $p_0$, we see that if $\ell$
and $\ell^{'}$ are any two lines through $p_0$ other than
$\ell_0$, the function $\chi_{\ell}-\chi_{\ell^{'}}$ lies in
$C(P,L_1)$. It is also in $ker\pi_{P_1}$. Thus, we can find $q-1$
linearly independent functions of this kind as described in the
statement. Then the dimension of $ ker\pi_{P_1} \cap C(P,L_1) $ is greater than or equal to
$q-1$. On the other hand, since none of the lines in $L_1$ has a
common point with $\ell_0$,  $C(P,L_1)$ is in the kernel of the
restriction map to $\ell_0$, while the image of the restriction of
$ ker\pi_{P_1} \cap C(P,L)$ to $\ell_0$ has dimension 2, spanned
by the images of $\chi_{\ell_0}$ and $\chi_{p_0}$. Thus, $
ker\pi_{P_1} \cap C(P,L_1)$ has codimension at least 2 in
$ker\pi_{P_1} \cap C(P,L)$, which has dimension $q+1$, by
Corollary 12. Hence, $$dim \left(ker\pi_{P_1} \cap C(P,L_1) \right) \leq q-1.$$

\end{proof}

\begin{cor}The spans of $Z \cup X_0$ and $L_1 \cup X_0$ are the same.
\end{cor}

\begin{proof} Let $\alpha$ be an element in the span of $L_1$. Since $Z$ maps to a basis of $C(P_1,L_1)$, there
is an element $\alpha^{'}$ in the span of $Z$ so that  $\pi_{P_1}(\alpha)=\pi_{P_1}(\alpha^{'}).$ Hence,
$\alpha - \alpha^{'} \in ker\pi_{P_1} \cap C(P,L_1)$.
By the previous lemma $ker\pi_{P_1} \cap C(P,L_1)$ is contained in the
span of $X_0$. Hence, we conclude that $\alpha$ is contained in the span of $X_0 \cup Z$.

\end{proof}

Therefore, $Z \cup X_0 \cup Y$ spans $C(P,L)$ as a vector space. So,
$dim(C(P,L))\leq dim(C(P_1,L_1))+2q$ and this implies $dim LU(3,q)
= q^3-dim(C(P,L))+2q.$ \\

\vspace{1cm}

\textsc{Acknowledgement:} \textit{I am grateful to Peter Sin for his constant support and encouragement. I would like to thank Stanley Payne for his interest and helpful remarks.I also would like to thank to Qing Xiang for his comments on the proof of lemma 8.}
\vspace{1cm}

\textsc{References:}\\

[1] B. Bagchi, A.E. Brouwer, and H.A. Wilbrink, \emph{Notes on binary
codes related to the O(5,q) generalized quadrangle for odd q},
Gemonetriae Dedicata, vol. 39, 1991 , pp. 339-355.

[2] D.B. Chandler, P. Sin, Q. Xiang, \emph{Incidence modules for symplectic spaces in characteristic two}, preprint, arXiv:math/0801.4392v1.

[3] R. G. Gallager, \emph{Low-density parity-check codes}, IRE
Trans. Inform. Theory, vol. IT-8, Jan. 1962, pp.21-28.

[4] J.-L. Kim, U. Peled, I. Pereplitsa, V. Pless, and S.
Friedland, \emph{Explicit construction of LDPC codes with no
4-cycles}, IEEE Trans. Inform. Theory, vol. 50, 2004, pp. 2378-2388.

[5] F. Lazebnik and V.A. Ustimenko, \emph{Explicit construction of
graphs with arbitrarily large girth and of size},
Discrete Applied Math., vol. 60, 1997, pp. 275-284.

[6] S.E. Payne, J.A. Thas, \emph{Finite Generalized Quadrangles}, Pittman Advanced Publishing Program,  Boston, London, Melbourne, 1984.

[7] N.S.N. Sastry , P. Sin, \emph{The code of a regular
generalized quadrangle of even order}, Group Representations:
Cohomology, Group Actions and Topology, ser. Proc. Symposia in
Pure Mathematics, vol. 63, 1998, pp. 485-496.

[8] P. Sin, Q. Xiang, \emph{On the dimensions of certain LDPC
codes based on q-regular bipartite graphs},  IEEE Trans.
Inform. Theory, vol. 52 (8), 2006, pp. 3735-3737.

\vspace{3cm}

\textsc{Department of Mathematics, University Of Florida, Gainesville, FL, 32611, USA}

\textit{E-mail Address}: ogul@math.ufl.edu

\end{document}